\documentclass[preprint,12pt]{elsarticle}

\usepackage{amssymb}
\usepackage{amsthm}
\usepackage{amsmath}

\journal{CAGD}

\usepackage{graphicx}
\usepackage{dcolumn}
\usepackage{bm}
\usepackage{amsmath}
\usepackage{amssymb}
\usepackage{latexsym}
\usepackage{epsfig}
\usepackage{amsbsy}
\usepackage{array}
\usepackage{amssymb}
\usepackage{setspace}
\usepackage{bm}

\def\sint{\ifmmode{- \!\!\!\!\!\! \int}
    \else{\hbox{$- \!\!\!\! \int \ $}}\fi}

\newtheorem{theorem}{Theorem}

\begin{document}

\begin{frontmatter}



\title{A note on convergence analysis of NURBS curve when weights approach infinity}

\author[label1]{Mao Shi}
\address[label1]{School of Mathematics and Information Science
of Shaanxi Normal University, Xi¡¯an 710062, China}

\cortext[cor1]{Email: shimao@snnu.edu.cn}

\begin{abstract}
This article considers the convergence of NURBS curve when weights approach infinity. We shows that limit of  NURBS curve dose not exist when independent variables weights approach infinity. Further, pointwise convergence uniform convergence and $L^1$ convergence are researched.
\end{abstract}


\end{frontmatter}


\section{Introduction}
NURBS curve is the key tools in Computer Aided Geometric Design (CAGD).
Let knot vector ${\mathbf{U}} = \{ u_0 ,...,u_m \}$ be a non-decreasing sequence of real numbers. The $p$th-degree NURBS curve on closed interval $[u_i,u_{i+1}]$ can be defined as
\[
{\mathbf{C}}(u) = \frac{{\sum\limits_{j = i - p}^i {N_{j,p} (u)\omega _j {\mathbf{p}}_j } }}
{{\sum\limits_{j = i - p}^i {N_{j,p} (u)\omega _j } }} = \sum\limits_{j = i - p}^i {R_{j,p}^\omega (u) {\mathbf{p}}_j }
\]
where $\{\mathbf{p}_j\}$ are control points, $\{\omega_j\}$ are weights, $\{N_{j,p}(u)\}$ are the $p$th-degree B-spline basis and $\{R_{j,p}^\omega(u)\}$ are  rational basis functions.

  The weights $\omega_j$ are typically used as shape parameters \cite{Farin}.  When ones of $\omega_j$ approach infinity, some results were obtain. For example, using weighted least squares, Piegl and Tiller\cite{Piegl} obtain that the NURBS curve  approaches  approximated data points and  the limit curve  consists of segment lines. Goldman \cite{Goldman} prove the NURBS curve will approaches the three control points.  Here, we  give another two results.

  \textbf{Example 1.} Let $\mathbf{U}=\{0,0,0,0,1,1,1,1\}$ and $\boldsymbol{\omega}=\{1,\omega_1, \omega_ 2, 1\}$. If $\omega_2 \asymp k\omega_1, k\neq 0$, we have the following result for cubic NURBS curve as $\omega_1 \rightarrow +\infty$.
  $$
  \lim_{\omega_1\to +\infty \atop \omega_2\to +\infty}\frac{{\sum\limits_{j = i - p}^i {N_{j,p} (u)\omega _j {\mathbf{p}}_j } }}
{{\sum\limits_{j = i - p}^i {N_{j,p} (u)\omega _j } }}=
\left\{ {\begin{array}{*{20}c}
   {{\mathbf{p}}_0 } & {u = u_3 }  \\
   {\frac{{N_1 (u){\mathbf{p}}_1  + kN_2 (u){\mathbf{p}}_2 }}
{{N_1 (u) + N_2 (u)}}} & {u \in (u_3 ,u_4 )}  \\
   {{\mathbf{p}}_3 } & {u = u_4 }  \\

 \end{array} }; \right.
$$
if  $\omega_2 \asymp k\omega_1^2$, we have the result for $\omega_1 \rightarrow +\infty$.
  $$
  \lim_{\omega_1\to +\infty \atop \omega_2\to +\infty}\frac{{\sum\limits_{j = i - p}^i {N_{j,p} (u)\omega _j {\mathbf{p}}_j } }}
{{\sum\limits_{j = i - p}^i {N_{j,p} (u)\omega _j } }}=
\left\{ {\begin{array}{*{20}c}
   {{\mathbf{p}}_0 } & {u = u_3 }  \\
    {{\mathbf{p}}_2 }& {u \in (u_3 ,u_4 )}  \\
   {{\mathbf{p}}_3 } & {u = u_4 }  \\

 \end{array} } \right..
$$

 In this paper, we using analysis methods to explain these results.
\section{Main results}
From the limit definition about multivariable and the \textbf{Example 1}, we can say that
\begin{theorem}
  The limit of NURBS curve does not exist when the $j \ (1<j \leq p+1)$ independent variables $\omega_j$ approach  infinity.
\end{theorem}
However, if the weights $\omega_j$  along an approach path we have the following theorems.
\begin{theorem}
  The NURBS curve ${\mathbf{C}}(u)$ on the interval $[u_i, u_{i+1}]$ is convergent pointwise when weights $\omega_j\rightarrow +\infty$.  
\end{theorem}
\begin{proof}
  For a fixed point $u\in [u_i, u_{i+1})$, dividing the weights $\{\omega_i\}$ into two sets $\boldsymbol{\omega_{j_0}}=\{\omega _{j_0 }\}$
 and $\boldsymbol{\omega_{j_1}}=\{\omega_{j_1}\}$, where elements in $\boldsymbol{\omega_{j_0}}$ have the same order of magnitude, that is,  $\frac{\omega_{j_0}}{k_{j_0}}\asymp min \ \boldsymbol{\omega_{j_0}} \ (j_0 \in \boldsymbol{j_0}) $, and are negligible with respect to ones in $\boldsymbol{\omega_{j_1}}$, we have

\[
\mathop {\lim }\limits_{{\boldsymbol{\omega_{j_0}}} \to \infty } {\mathbf{C}}(u) = \left\{ {\begin{array}{*{20}c}
   {{{\sum\limits_{j \in {\mathbf{j}}_0  - \{ i + 1\} } {N_{j,p} (u)k_j {\mathbf{p}}_j } } \mathord{\left/
 {\vphantom {{\sum\limits_{j \in {\mathbf{j}}_0  - \{ i + 1\} } {N_{j,p} (u)k_j {\mathbf{p}}_j } } {\sum\limits_{j \in {\mathbf{j}}_0  - \{ i + 1\} } {N_{j,p} (u)k_j } }}} \right.
 \kern-\nulldelimiterspace} {\sum\limits_{j \in {\mathbf{j}}_0  - \{ i + 1\} } {N_{j,p} (u)k_j } }}} & {u = u_i }  \\
   {{{\sum\limits_{j \in {\mathbf{j}}_0 } {N_{j,p} (u)k_j {\mathbf{p}}_j } } \mathord{\left/
 {\vphantom {{\sum\limits_{j \in {\mathbf{j}}_0 } {N_{j,p} (u)k_j {\mathbf{p}}_j } } {\sum\limits_{j \in {\mathbf{j}}_0 } {N_{j,p} (u)k_j } }}} \right.
 \kern-\nulldelimiterspace} {\sum\limits_{j \in {\mathbf{j}}_0 } {N_{j,p} (u)k_j } }}} & {u \in (u_i ,u_{i + 1} )}  \\
   {{{\sum\limits_{j \in {\mathbf{j}}_0  - \{ i\} } {N_{j,p} (u)k_j {\mathbf{p}}_j } } \mathord{\left/
 {\vphantom {{\sum\limits_{j \in {\mathbf{j}}_0  - \{ i\} } {N_{j,p} (u)k_j {\mathbf{p}}_j } } {\sum\limits_{j \in {\mathbf{j}}_0  - \{ i\} } {N_{j,p} (u)k_j } }}} \right.
 \kern-\nulldelimiterspace} {\sum\limits_{j \in {\mathbf{j}}_0  - \{ i\} } {N_{j,p} (u)k_j } }}} & {u = u_{i + 1} }  \\

 \end{array} ,} \right.
\]
 which is the desired conclusion.
\end{proof}
\begin{theorem}
  If the knot vector ${\mathbf{\tilde{U}}} = \{ u_0 ,...,u_m \} $ be a strictly increasing sequence, where $i-p \notin \boldsymbol{{j_0}} \ \wedge  i \notin \  \boldsymbol{{j_0}}$, the NURBS curve on the interval $[u_i,u_{i+1}]$ is uniformly convergent. the limit curve is as following
  $$\mathop {\lim }\limits_{{\boldsymbol{\omega_{j_0}}} \to \infty } {\mathbf{C}}(u) ={{{\sum\limits_{j \in {\mathbf{j}}_0 } {N_{j,p} (u)k_j {\mathbf{p}}_j } } \mathord{\left/
 {\vphantom {{\sum\limits_{j \in {\mathbf{j}}_0 } {N_{j,p} (u)k_j {\mathbf{p}}_j } } {\sum\limits_{j \in {\mathbf{j}}_0 } {N_{j,p} (u)k_j } }}} \right.
 \kern-\nulldelimiterspace} {\sum\limits_{j \in {\mathbf{j}}_0 } {N_{j,p} (u)k_j } }}}, \ u\in [u_i,u_{i+1}].$$
\end{theorem}
\begin{proof}
  Without loss of generality,  consider one weight $\omega_k$ $(k\neq i-p \wedge k \neq i)$ approaches positive infinity. When $\omega_k > \Omega (\varepsilon)=\frac{1}{\varepsilon}\frac{M}{ m}$, where $$M=max \sum\limits_{j = i - p \wedge j \ne k}^i {N_{j,p} (u)\omega _j }, \ \ m= min N_{k,p}(u), $$
   we have
  \[
\left| {R_{j,p}^\omega   - 1} \right| = \left| {\frac{{\sum\limits_{j = i - p \wedge j \ne k}^i {N_{j,p} (u)\omega _j } }}
{{\sum\limits_{j = i - p}^i {N_{j,p} (u)\omega _j } }}} \right| < \left| {\frac{{\sum\limits_{j = i - p \wedge j \ne k}^i {N_{j,p} (u)\omega _j } }}
{{N_{k,p} (u)\omega _k }}} \right| < \varepsilon
\]
 or equivalently,
\[
\omega _k  > \frac{1}
{\varepsilon }\frac{{\sum\limits_{j = i - p \wedge j \ne k}^i {N_{j,p} (u)\omega _j } }}
{{N_{k,p} (u)}} = \frac{1}
{\varepsilon }\frac{M}
{m},
\]
where $N_{k,p}(u)\neq 0$ on the interval $[u_i,u_{i+1}]$ of knot vector ${\mathbf{\tilde{U}}}$. By the properties of continuous function on the closed interval, $M$ and $m$ do not depend on the parameter $u$, which imply the $R_{j,p}^\omega$  is uniformly convergent and then the NURBS curve is also uniformly convergent.
\end{proof}

Finally, with the convex hull property of NURBS curve and Bounded convergence theorem \cite{Stein}, we obtain

\begin{theorem}
  The NURBS curve ${\mathbf{C}}(u)$ on the interval $[u_i, u_{i+1}]$ is $L^1$ convergence when  $\boldsymbol{\omega_{j_0}}\rightarrow +\infty$.
\end{theorem}

\section*{Conclusion}
We analysis convergence of NURBS curve in this paper. The geometric meaning of limit curve is our future research.

\section*{Acknowledgements}
\label{}
This work is supported by the Natural Science Basic Research Plan in Shaanxi Province of China (No. 2013JM1004) and the Fundamental Research Funds for the Central Universities (No. GK201102025).



\end{document}